\documentclass{amsart}

\usepackage{amssymb,amsmath}
\usepackage{color}
\usepackage{esint}

\theoremstyle{plain}
\newtheorem{theorem}{Theorem}[section]
\newtheorem*{theorem*}{Theorem}

\newtheorem{lemma}[theorem]{Lemma}

\theoremstyle{remark}
\newtheorem{remark}[theorem]{Remark}
\numberwithin{equation}{section}
\theoremstyle{definition}
\newtheorem{definition}[theorem]{Definition}
\numberwithin{equation}{section}

\newcount\quantno
\everydisplay{\quantno=0}\everycr{\quantno=0}
\newcommand\quant{\advance\quantno by1
                      \ifnum\quantno=1\qquad\else\quad\fi\forall }
\newcommand\itemno[1]{(\romannumeral #1)}

\newcommand\rest[1]{\kern-.1em
          \lower.5ex\hbox{$\scriptstyle #1$}\kern.05em}

\newcommand\set[1]{{\left\{#1\right\}}}

\renewcommand\mod[1]{\vert{#1}\vert}
\newcommand\bigmod[1]{\bigl\vert{#1}\bigr|}

\newcommand\norm[2]{{\Vert{#1}\Vert_{#2}}}

\newcommand\bignorm[2]{\left.{\big\Vert{#1}\big\Vert_{#2}}\right.}
\newcommand\bignormto[3]{{\big\Vert{#1}\big\Vert^{#3}_{#2}}}
\newcommand\Bignorm[2]{\left.{\Big\Vert{#1}\Big\Vert_{#2}}\right.}

\newcommand\prodo[2]{\left\langle#1,#2\right\rangle}

\newcommand\wrt{\,\text{\rm d}}

\newcommand\Riesz{{\operatorname{\mathcal{R}}^d}}
\newcommand\RieszOne{{\operatorname{\mathcal{R}}^1}}
\newcommand\shRiesz{{\operatorname{\mathcal{R}}_c^d}}
\newcommand\scalarshRiesz{{\operatorname{\mathcal{R}}_{c,\cZ}^d}}

 \newcommand\br{\mathbf{r}}

\newcommand\bW{\mathbf{W}}

\newcommand\BC{\mathbb{C}}

\newcommand\BR{\mathbb{R}}

\newcommand\BX{\mathbb{X}}

\newcommand\cA{\mathcal{A}}   
\newcommand\fra{\mathfrak{a}}  
   
\newcommand\frb{\mathfrak{b}} 
   
\newcommand\frc{\mathfrak{c}} 
\newcommand\cD{\mathcal{D}}

                               \newcommand\frg
{\mathfrak{g}} 
\newcommand\cH{\mathcal{H}}   
\newcommand\frh{\mathfrak{h}}

\newcommand\frk{\mathfrak{k}} 
\newcommand\cL{\mathcal{L}}

\newcommand\frn{\mathfrak{n}}

\newcommand\frp{\mathfrak{p}}  

\newcommand\cR{\mathcal{R}}   
 
\newcommand\cS{\mathcal{S}}

\newcommand\cZ{\mathcal{Z}}

\newcommand\KGK{{K \backslash G/K}}

\newcommand\al{\alpha}
\newcommand\be{\beta}
\newcommand\ga{\gamma}    \newcommand\Ga{\Gamma}
\newcommand\de{\delta}
  \newcommand\vep{\varepsilon}

\newcommand\la{\lambda}

\newcommand\si{\sigma} \newcommand\Si{\Sigma} 

\newcommand\vp{\varphi}

\newcommand\OV{\overline}
\newcommand\funnyk{k\hbox to 0pt{\hss\phantom{g}}}

\newcommand\lu[1]{L^1(#1)}
\newcommand\lp[1]{L^p(#1)}

\newcommand\ld[1]{L^2(#1)}

\newcommand\hu[1]{H^1(#1)}
\newcommand\ghu[1]{\frh^1(#1)}

\newcommand\floor[1]{\lfloor#1\rfloor}

\newcommand\bc{\mathbf{c}}

\newcommand\wt{\widetilde}
\newcommand\whH{\widehat{\phantom{G}}\hbox to 0pt{\hss $H$}}
\newcommand\ad{\text{ad}}
\newcommand\Ad{\text{Ad}}

\newcommand\emspace{\hbox to 6pt{\hss}}
\newcommand\ds{\displaystyle}

\newcommand\Xu[1]{X^1(#1)}

\newcommand\Xh[1]{X^k(#1)}

{\mathsurround=0pt}

\newcommand\rmi{\hbox{\rm (i)}}
\newcommand\rmii{\hbox{\rm (ii)}}
\newcommand\rmiii{\hbox{\rm (iii)}}

\newcommand\One{{\mathbf{1}}}

\newcommand\e{\mathrm{e}}

\newcommand\sft[1]{\wt{#1}}

\newcommand\planc{\mod{\mathbf{c}(\la)}^{-2} \wrt \la}

\newcommand\supp{\operatorname{\mathrm supp}}

\DeclareSymbolFont{EUEX}{U}{euex}{m}{n}

\DeclareSymbolFont{euexlargesymbols}{U}{euex}{m}{n}
\DeclareMathSymbol{\intop}{\mathop}{euexlargesymbols}{"52}
     \def\int{\intop\nolimits}

\DeclareSymbolFont{euexsymbols}     {U}{euex}{m}{n}

\title[Higher order Riesz transforms]
{Higher order Riesz transforms \\
on noncompact symmetric spaces}

\author[G. Mauceri, S. Meda and M. Vallarino]
{Giancarlo Mauceri, Stefano Meda and Maria Vallarino}

\thanks{Work partially supported by PRIN 2015 ``Real and complex manifolds: 
geometry, topology and harmonic analysis". The authors are members of the 
Gruppo Nazionale per l'Analisi Matematica, la Probabilit\`a e 
le loro Applicazioni (GNAMPA) of the Istituto Nazionale di Alta Matematica (INdAM)
}

\address{Giancarlo Mauceri \\
Dipartimento di Matematica \\ Universit\`a di Genova\\
via Dodecaneso 35, 16146 Genova, \\ Italy}
\email{mauceri@dima.unige.it}

\address{Stefano Meda\\
Dipartimento di Matematica e Applicazioni \\ 
Universit\`a di Milano-Bicocca\\
via R.~Cozzi 53,  I-20125 Milano, \\ Italy}
\email{stefano.meda@unimib.it}

\address{Maria Vallarino \\
Dipartimento di Scienze Matematiche ``Giuseppe L. Lagrange" \\ 
Politecnico di Torino\\ Corso Duca degli Abruzzi 24, \\ I-10129 Torino, \\ Italy}
\email{maria.vallarino@polito.it}



\begin{document}

\begin{abstract}
In this note we prove various sharp boundedness results on suitable Hardy type spaces for 
Riesz transforms of arbitrary order on noncompact symmetric spaces of arbitrary rank.
\end{abstract}

\maketitle

\setcounter{section}{0}
\section{Introduction} \label{s: Introduction}

Suppose that $\BX$ is a Riemannian symmetric spaces of the noncompact type,
and denote by~$\nabla$ the covariant derivative and by~$\cL$ the Laplace--Beltrami operator on $\BX$.  
The purpose of this paper is to prove estimates 
on the Hardy type spaces $\hu{\BX}$ and $\Xh{\BX}$  
for the (higher order) \emph{Riesz transform} $\Riesz$, defined, for any positive integer~$d$,~by
\begin{equation} \label{f: shRiesz Riesz}
\Riesz
= \nabla^d \cL^{-d/2},  
\end{equation} 
and the \emph{shifted Riesz transform}
$\shRiesz$, where $c>0$, defined by 
\begin{equation} \label{f: shRiesz Riesz}
\shRiesz 
= \nabla^d (\cL+c)^{-d/2}.  
\end{equation} 
Here $\hu{\BX}$ is the space
introduced by A.~Carbonaro, Mauceri and Meda in \cite{CMM}, and $\Xh{\BX}$
denotes the space introduced in \cite{MMV}, and further investigated in the series of papers
\cite{MMV3,MMVdual,MMV5}.  

Riesz transforms on $\BX$ have been the object of 
a number of investigations in the last forty years, or so.  Without any pretence of 
exhaustiveness, we recall the works of J.-Ph.~Anker
and his collaborators \cite{A1,AJ} and of N.~Lohou\'e \cite{L1,L2}. 
For more on the analysis of Riesz transforms on a wider class of
Riemannian manifolds with spectral gap and bounded geometry, see \cite{MMV3,MMVdual,MMV5}
and the references therein.   ALTRO? 

To the best of our knowledge, there are very few endpoint results for $p=1$ 
for (higher order) Riesz transforms on nondoubling Riemannian manifolds.  
Nonshifted Riesz transforms of order $d$ are known to be bounded
from $\lp{\BX}$ to the space $\lp{\BX;T^d}$ of all $p$-integrable 
covariant tensors of order $d$ on $\BX$ for every $p$ in $(1,\infty)$ \cite{A1,AJ}. 
See also \cite{L1}, where a similar result is proved for Riesz transforms of
even order on certain Cartan--Hadamard manifolds. 
Anker \cite{A1} also proved that the first and second order Riesz transforms
are of weak type~$1$, and observed that this is no longer true of Riesz transforms of order 
$\geq 3$, at least in the rank one case.  
The same is presumably true in any noncompact symmetric space.  
Anker's proof relies on very fine estimates of the heat kernel and its derivatives.  
So does the proof of our main result concerning $\Riesz$.  
In particular, the key point in our argument is to obtain good estimates of the kernel of the operator 
$\nabla^d \cL^{1/2}$ (see Lemma~\ref{l: B4Bc}).   
We are not able to prove similar estimates on a more general class of manifolds:
this is the reason for which we restrict our analysis to the case of noncompact symmetric spaces.    
We prove the following:
\begin{enumerate}
\item[\itemno1] suppose that $d$ is a positive integer.  Then
\begin{enumerate}
\item[(a)]
$\shRiesz$ is bounded from $\hu{\BX}$ to $\lu{\BX;T^d}$; 
\item[(b)]
$\Riesz$ is bounded from $\Xh{\BX}$ to $\lu{\BX;T^d}$ 
for every $k \geq \floor{(d+1)/2}$;  
\end{enumerate}
\item[\itemno2]
$\RieszOne$ is unbounded from $\hu{\BX}$ to $\lu{\BX;T^1}$.  
\end{enumerate}
Part \rmi\ is proved in Theorem~\ref{t: RT} and part~\rmii\ in Theorem~\ref{t: notH^1L^1}. 
The remarkable difference between the boundedness properties of $\shRiesz$ and $\Riesz$
on $\hu{\BX}$ and $\Xh{\BX}$  has a simple explanation.  
Fix a base point $o$ in $\BX$.  On the one hand, the ``convolution
kernels'' of $\shRiesz$ and $\Riesz$ have a similar
behaviour in a neighbourhood of $o$, where they are homologous to a kernel of a  
standard singular integral operator.  
On the other hand, the kernel of $\shRiesz$ is integrable at infinity, 
whereas that of $\Riesz$ is not. 
Furthermore, the greater the order $d$ is, the slower decay the kernel 
of $\Riesz$ has at infinity.   


We emphasize that the results in this paper 
aim at corroborating the fact that $\Xh{\BX}$ does serve as an effective counterpart 
on $\BX$ of the classical Hardy space $\hu{\BR^n}$, 
whereas the effectiveness of the space $\hu{\BX}$ is 
somewhat limited to operators whose kernels are integrable at infinity.    
It is important to keep in mind that 
the following strict continuous containments hold   
$$
\hu{\BX}
\supset \Xu{\BX} \supset X^2(\BX) \supset \cdots \supset \Xh{\BX} \supset \cdots
$$    
There is a huge literature concerning Hardy type spaces on Riemannian manifolds, and 
sometimes a bit of confusion about their effectiveness.  
As already mentioned, we prove that shifted Riesz transforms of order $d$ are bounded
from $\hu{\BX}$ to $\lu{\BX,T^d}$. 
We would like to make it clear that
we could have proved similar results
involving Taylor's version of Goldberg's Hardy space $\ghu{\BX}$ \cite{T2} or the space of 
Carbonaro, A.~McIntosh and A.~Morris \cite{CMcIM} instead of the space $\hu{\BX}$.  
All these spaces may be used to give endpoint results for $p=1$ for operators with 
kernels that are integrable at infinity and behave, roughly speaking, as Calder\'on--Zygmund 
operators near the origin.  
However, none of these is apt to serve as an endpoint result for nonshifted Riesz transforms
(of whatsoever order)
in the setting of noncompact symmetric spaces.  
In our paper we show that the nonshifted Riesz transform $\RieszOne$
does not map $\hu{\BX}$ into $\lu{\BX}$.  \textit{A fortiori},
it cannot map Taylor's space $\ghu{\BX}$, or the Carbonaro--McIntosh--Morris space, into $\lu{\BX}$.  

The paper is organised as follows.  Section~\ref{s: symmetricspacesH1} contains the basic 
notions of analysis on~$\BX$ and the definitions of the Hardy spaces $H^1(\BX)$ and $X^k(\BX)$.
In Section~\ref{s: Riesz} we prove the positive results for the 
Riesz transforms (see Theorem \ref{t: RT}).  
Finally, in the last section we prove that the Riesz 
potentials $\cL^{-\tau}$, $\tau>0$, and the first order Riesz transform 
$\RieszOne$ are unbounded from $H^1(\BX)$ to $L^1(\BX)$.  

We will use the ``variable constant convention'', and denote by $C,$
possibly with sub- or superscripts, a constant that may vary from place to
place and may depend on any factor quantified (implicitly or explicitly)
before its occurrence, but not on factors quantified afterwards. 

\section{Preliminaries}\label{s: symmetricspacesH1}
 
\subsection{Preliminaries on symmetric spaces}\label{s: symmetricspaces}  

In this subsection we recall the basic notions of analysis on noncompact 
symmetric spaces that we shall need in the sequel. Our main 
references are the books \cite{H1, H2} and the papers \cite{A2,A1, AJ}. 
For the sake of the reader we recall also the notation, which is quite 
standard. 

We denote by $G$ a noncompact connected real semisimple Lie group
with finite centre, by $K$ a maximal compact subgroup and by $\BX=G/K$ 
the associated noncompact Riemannian symmetric space. 
The point $o=e{K}$, where $e$ is the identity of $G$, is called the \emph{origin} in $\BX$. 
Let $\theta$ and $\frg=\frk\oplus\frp$ be 
the corresponding Cartan involution and Cartan decomposition of the  Lie 
algebra $\frg$ of $G$, and $\fra$ be a maximal abelian subspace of $\frp$. 
We denote by $\Si$ the restricted root system of $(\frg,\fra)$ and by $W$ the 
associated Weyl group.  Once a positive Weyl chamber $\fra^+$ has 
been selected, $\Si^+$ denotes the corresponding set of positive 
roots, $\Si_s$ the set of simple roots in $\Si^+$ and $\Si_0^+$ 
the set of positive indivisible roots. As usual, 
$\frn=\sum_{\al \in \Si^+} \frg_{\al}$ denotes
the sum of the positive root spaces.  Denote by $m_\al$ the dimension of $\frg_{\al}$
and set $\rho := (1/2) \sum_{\al \in \Sigma^+} \, m_\al\, \al$.  
We denote by $\bW$ the interior of the convex hull
of the points $\set{w \cdot \rho : w \in W}$. Clearly $\bW$ is an open 
convex polyhedron in $\fra^*$.
By $N=\exp \frn$ and  $A=\exp \fra$  we denote the analytic 
subgroups of $G$ corresponding to $\frn$ and $\fra$. 
The Killing form $B$ induces the $K$--invariant inner 
product $\prodo{X}{Y} = -B\big(X,\theta(Y)\big)$ on $\frp$ and 
hence a $G$--invariant metric $d$ on $\BX$. 
The ball with centre $x\cdot o$ and radius $r$ will be denoted by $B_r(o)$.  
The map $X\mapsto \exp X \cdot o$ is a diffeomorphism of $\mathfrak{p}$ onto $\BX$. 
The distance of $\exp X\cdot o$ from the origin in $\BX$ is equal to $\mod{X}$,
and will be denoted by $\mod{\exp X\cdot o}$ .  We denote by $n$ the dimension of $\BX$ and 
by $\ell$ its rank, i.e. the dimension of $\fra$. 

We identify functions on the symmetric space $\BX$ with
$K$--right-invariant functions on $G,$ in the usual way. If $E(G)$ 
denotes a space of functions on $G$, we
define $E({\BX})$  and $E({K\backslash{\BX}})$ to be the closed 
subspaces of $E(G)$ of the $K$--right-invariant and the $K$--bi-invariant functions, respectively. 
If $D=Z_1Z_2\cdots Z_d$, with $Z_i\in\frg$, then we denote by $Df(x)$ the right
differentiation of $f$ at the point $x$ in $G$.  Thus,
\begin{align*}
Df(x)
& = \frac{\partial^d}{\partial t_1\cdots\partial t_d}
      f\big(x\exp(t_1Z_1)\cdots\exp(t_d Z_d)\big)_{\big\vert_{t_1=\ldots=t_d=0}}.
\end{align*}
We write $\wrt x$ for a Haar measure on $G$, and let $\wrt k$ be the
Haar measure on $K$ of total mass one.  
The Haar measure of $G$ induces a $G$--invariant measure $\mu$ on $\BX$ for which
$$
\int_{\BX} f(x\cdot o) \, \wrt\mu (x\cdot o) 
= \int_G f(x) \, \wrt x
\quant f \in C_c({\BX}).
$$
We shall often write $\bigmod{E}$ instead of $\mu(E)$ for a measurable subset $E$ of $\BX$.  
We recall that
\begin{equation}\label{intCartan}
\int_G f(x)\wrt x 
= \int_K\!\int_{\fra^+}\!\int_K 
f\bigl(k_1\exp H \,k_2\bigr)\, \de(H) \wrt k_1 \wrt H \wrt k_2 \,,
\end{equation}
where $\wrt H$ denotes a suitable nonzero multiple of the Lebesgue
measure on $\fra$, and 
\begin{equation}\label{delta}
\delta(H)=
\prod_{\al \in\Sigma^+} \big(\sinh\al(H)\big)^{m_{\al}} 
\leq C\, \e^{2\rho(H)}
\quant H \in \fra^+.
\end{equation}
We recall the Iwasawa decomposition of $G$, which is $G=K\,A\,N$. 
For every $x$ in $G$ we denote by 
$H(x)$ the unique element of $\mathfrak{a}$ such that $x\in K\exp H(x) N$.
For any linear form $\la: \fra \to \BC$, 
the elementary spherical function $\vp_\la$ is defined by the rule
$$
\vp_\la (x) = \int_K \e^{-(i\la + \rho) H(x^{-1}k)} \wrt k
\quant x \in G.
$$
In the sequel we shall use the following estimate of the spherical function 
$\vp_0$ \cite[Proposition 2.2.12]{AJ}:
\begin{equation}\label{phi0}
\vp_0(\exp H\cdot o)
\leq (1+\mod{H})^{|\Sigma_0^+|}\,\,  \e^{-\rho(H)}\quant H\in\fra^+.
\end{equation}
The spherical transform $\cH f$ of an $\lu{G}$ function $f$, 
also denoted by $\sft f$, is defined by the formula
$$
\cH f(\la) 
= \int_G f (x) \, \phi_{-\la} (x) \wrt x
\quant \la \in \fra^*.
$$
Harish-Chandra's inversion formula and Plancherel formula state
that for ``nice'' $K$--bi-invariant functions $f$ on $G$
\begin{equation} \label{f: inversion}
f(x) = \int_{\fra^*} \sft f (\la) \, \phi_\la(x)\, \wrt\nu(\la )
\quant x \in G
\end{equation}  
and
$$
\norm{f}2 = \left[\int_{\fra^*} \mod{\sft f(\la)}^2\, 
\wrt\nu(\la ) \right]^{1/2}
\quant f\in \ld {\KGK},
$$
where $\wrt\nu(\la ) = c_{{}_G} \planc $, and $\bc$ denotes 
the Harish-Chandra $\bc$-function.  
We do not need the exact form of $\bc$.  It will be enough to know that there
exists a constant~$C$ such that 
\begin{equation}\label{f: estimatec}
\mod{\bc(\la)}^{-2} 
\leq C\, \bigl(1+\mod{\la}   \bigr)^{n-\ell},
\end{equation}
\cite[IV.7]{H1}.  

Next, we recall the Cartan decomposition of $G$, which is $G=K\exp \OV{\fra^+}K$.  
In fact, for almost every~$x$ in $G$, 
there exists a unique element $A^+(x)$ in $\fra^+$ such that 
$x$ belongs to $K \exp A^+(x) K$. 

\begin{lemma} \label{l: Huc}
The map $A^+:G\to \fra$ is Lipschitz with respect to both 
left and right translations of $G$.  More precisely
$$
\mod{A^+(yx)-A^+(y)}\le  \ d(x\cdot o,o)\qquad \text{and}\qquad \mod{A^+
(xy)-A^+(y)}\le \ d(x\cdot o,o)\,,
$$
for all $x$ and $y$ in $G$.
\end{lemma}

\begin{proof}
The first inequality follows from $\mod{A^+(yx)-A^+(y)}\le  \ d(yx\cdot o,y
\cdot o)$, see \cite[Lemma 2.1.2]{AJ}, and the $G$--invariance of the 
metric $d$ on $\BX$. 

The second inequality follows from the first, 
for $A^+(x^{-1})=-\si\,A^+(x)$, where $\si$ is the element of the Weyl 
group that maps the negative Weyl chamber $-{\fra^+}$ to 
the positive Weyl chamber~${\fra^+}$.
\end{proof}

\noindent
For every positive $r$ we define
\begin{equation} \label{f: frbj}
\frb_r  = \{H\in\fra: \mod{H}\leq r\}    
\qquad\hbox{and}\qquad  B_r   = K(\exp \frb_r)K.  
\end{equation}
The set $B_r$ is the inverse image under the canonical projection $\pi:G \to \BX$ 
of the ball $B_r(o)$ in the symmetric space $\BX$. 
Thus, a function $f$ on $\BX$ is supported in $B_r(o)$ if and only if, as a 
$K$--right-invariant function on $G$, is supported in $B_r$. 

\subsection{Hardy spaces on $\BX$}\label{s: H1}

In this subsection we briefly recall the definitions and properties 
of $H^1(\BX)$ and $X^k(\BX)$. 
For more about $H^1(\BX)$ and $X^k(\BX)$ we refer the reader to \cite{CMM}
and \cite{MMV, MMV3, MMVdual}, respectively.

\begin{definition}
An \emph{$H^1$-atom} is a function $a$ in $\ld{\BX}$, with support
contained in a ball~$B$ of \emph{radius at most $1$}, and such that  
\begin{enumerate}
\item[\itemno1]
$\int_B a \wrt \mu  = 0$;
\item[\itemno2]
$\norm{a}{2}  \leq \bigmod{B}^{-1/2}$.
\end{enumerate}
\end{definition}
\begin{definition}
The \emph{Hardy space} $H^{1}({\BX})$ is the space of all functions~$g$ in $\lu{\BX}$
that admit a decomposition of the form
\begin{equation} \label{f: decomposition}
g = \sum_{j=1}^\infty c_j\, a_j,
\end{equation}
where $a_j$ is an $H^1$-atom, and $\sum_{j=1}^\infty \mod{c_j} < \infty$.
Then $\norm{g}{H^{1}}$ is defined as the infimum of $\sum_{j=1}^\infty \mod{c_j}$
over all decompositions (\ref{f: decomposition}) of $g$.  
\end{definition}

\begin{remark} \label{rem: economical dec}
A straightforward consequence of \cite[Lemma~5.7]{MMV} that we shall use repeatedly in the sequel is the following.  If $f$ is in $\ld{\BX}$, 
its support is contained in $B_R(o)$ for some $R>1$, and its integral vanishes, 
then $f$ is in $\hu{\BX}$, and 
$$
\bignorm{f}{H^1}
\leq C \, R \, \bigmod{B_R(o)}^{1/2} \, \bignorm{f}{2}.  
$$
\end{remark}

\noindent
The Hardy type spaces $\Xh{\BX}$ were introduced in \cite{MMV} as certain Banach 
spaces isometrically isomorphic to $\hu{\BX}$.  An atomic characterisation
of $\Xh{\BX}$ was then established in \cite{MMV3}, and refined in \cite{MMVdual}.   
In this paper we adopt the latter as the definition of~$\Xh{\BX}$.   
We say that a (smooth) function $Q$ on $\BX$ is $k$--quasi-harmonic if $\cL^k Q$
is constant on~$\BX$.  

\begin{definition}
Suppose that $k$ is a positive integer. 
An \emph{$X^k$-atom} is a function $A$, with support contained in a ball $B$
of \emph{radius at most $1$}, such that
\begin{enumerate}
\item[\itemno1]
$\int_{\BX} A \,\, Q \wrt \mu = 0$ for every $k$--quasi-harmonic function $Q$;
\item[\itemno2]
$\ds\norm{A}{2}\leq \bigmod{B}^{-1/2}$.
\end{enumerate}
Note that condition \rmi\ implies that 
$\int_{\BX} A\wrt\mu=0$, because the constant function $1$ is~$k$--quasi-harmonic
on $\BX$. 
\end{definition}

\begin{definition} 
The space $X^k(\BX)$ is the space of all functions $F$ of the 
form $\sum_j c_j\, A_j$, where $A_j$ are $X^k$-atoms and 
$\sum_j \mod{c_j} < \infty$, endowed with the norm
$$
\norm{F}{X^k}
= \inf\, \bigl\{\sum_{j} \mod{c_j}:  F = \sum_{j} c_j \, A_j,\quad
\hbox{where $A_j$ is an $X^k$-atom}\bigr\}. 
$$
\end{definition} 

\subsection{Estimate of operators}
We shall encounter various occurrences of the problem of estimating the $\hu{\BX}$ norm 
of functions of the form $a\ast\ga$, where $a$ is an $\hu{\BX}$-atom with support in
$B_R(o)$ for some $R\leq 1$, and $\ga$ is a $K$--bi-invariant function with support contained
in the ball $\OV{B_{\be}(o)}$.
The following lemma contains a version of such an estimate that we shall use
frequently in the sequel.  For the proof, see \cite{}.  DECIDERE A COSA RIFERIRSI

\begin{lemma} \label{l: est on annuli}
Suppose that $a$ and $\ga$ are as above.  The following hold:
\begin{enumerate}
\item[\itemno1]
there exists a constant $C$ such that 
$$
\norm{a\ast \ga}{H^1}
\leq 
\begin{cases}
\bigmod{B_{R+\be}(o)}^{1/2} \, 
      \min\big(\bignorm{\ga}{2}, C\, R\, \bignorm{\nabla\ga}{2} \big)
                                   & \hbox{if $R+\be \leq 1$}  \\ 
C\, (R+\be) \, \bigmod{B_{R+\be}(o)}^{1/2} \, \bignorm{\ga}{2} 
                                   & \hbox{if $R+\be > 1$};
\end{cases}
$$
\item[\itemno2]
suppose further that $\ga$ is of the form $\cA^{-1}(\Phi\cA\kappa)$, where $\Phi$
is a smooth function with compact support, and define
$s:= (n-\ell)/2$.  Then there exists a constant~$C$ such that 
$$
\bignorm{\cA^{-1}(\Phi \cA\kappa)}{2}
\leq C \bignorm{\Phi\cA\kappa}{H^s(\fra)}    
$$
and
$$
\bignorm{\nabla\big[\cA^{-1}(\Phi \cA\kappa)\big]}{2}
\leq C \bignorm{\Phi\cA\kappa}{H^{s+1}(\fra)},    
$$
where $H^s(\fra)$ denotes the standard Sobolev space of order $s$ on $\fra$.  
\end{enumerate}
\end{lemma}

\section{Riesz transforms}\label{s: Riesz}

Our analysis of Riesz transforms may be reduced to that of certain 
operators, called scalar Riesz transforms, which are 
convolution operators whose kernels are smooth functions on $\BX\setminus\set{o}$. 
To describe these kernels we need more notation. 

For $y$ in $G$ we denote by $L(y)$ left translation by~$y$ acting on $G$, by $\wrt L(y)$ and by 
$L(y)^*$ the differential and the pull-back of 
$L(y)$ acting on tangent vectors and covariant tensors on $G$, respectively. 
With a slight abuse of notation we shall also denote by 
$L(y)$, $\wrt L(y)$ and $L(y)^*$ the corresponding maps, acting on 
$\BX$, on tangent vectors and on covariant tensors on~$\BX$.  
Thus $L(y)^*$ is an isometry between covariant tensors at the point 
$y\cdot o$ to covariant tensors at the point~$o$. 
We recall that the tangent space of~$\BX$ at the point~$o$ is 
identified with $\frp$ and the space of covariant tensors of order 
$d$ at the point $o$ is identified with~$(\frp^*)^{\otimes d}$. 

For every $\cZ=(Z_1,\ldots,Z_d)$ in $\frp^d$ 
the \emph{scalar shifted Riesz transform} $\scalarshRiesz$ of order $d$ is the operator defined by
\begin{equation} \label{f: scalar Riesz}
\scalarshRiesz f(x\cdot o)
= L(x)^*\big[\shRiesz f(x\cdot o)\big] (Z_1,\ldots,Z_d)
\quant x\in G.
\end{equation} 
It is well known that $\Riesz$ is bounded from $\ld{\BX}$
to $\ld{\BX;T^d}$ \cite{Str,Au}.  A straightforward argument shows that the same is
true of $\shRiesz$, and, consequently, $\scalarshRiesz$ is bounded on $\ld{\BX}$
for every $\cZ$ in $\frp^d$. 
Our endpoint result for Riesz transforms is the following. 

\begin{theorem}\label{t: RT}
Suppose that $d$ is a positive integer and $c>0$. The following hold:
\begin{enumerate}
\item[\itemno1] 
for every $\cZ$ in $\frp^d$ the operator $\scalarshRiesz$ extends to a bounded 
operator on $\hu{\BX}$;
\item[\itemno2] 
$\shRiesz$ extends to bounded operator from $\hu{\BX}$ to $\lu{\BX;T^d}$;
\item[\itemno3] 
$\Riesz$ extends to a bounded operator from 
$X^{\floor{(d+1)/2}}(\BX)$ to $\lu{\BX;T^d}$.
\end{enumerate}
\end{theorem}
\par\noindent
For every $z\in \BC$ and $c\geq 0$,
the Bessel--Riesz potential $(\cL+c)^{-z/2}$ maps the space of test functions 
$\cD(\BX)$ into the space of distributions $\cD'(\BX)$. 
If  $z\not=0,-2,-4,\ldots$, then its convolution kernel $\kappa_c^z$ is a distribution, which, 
away from the origin $o$, coincides with the function
\begin{equation} \label{f: kappacz}
\kappa_c^z(x\cdot o)
= \frac{1}{\Ga(z/2)}\int_{0}^{\infty}t^{z/2-1}\, \e^{-ct} \,h_{t}(x\cdot o) \,\wrt t,
\end{equation}
where $h_t$ denotes the heat kernel on $\BX$. 
In the sequel we shall use repeatedly the estimates 
for $\kappa_c^z$ and their derivatives obtained by 
Anker and Ji \cite[Thm 4.2.2]{AJ}.
They actually considered the case where $z\ge 0$, 
but their arguments extend almost \emph{verbatim} to all complex $z\not=0,-2,-4,\ldots$; 
in particular, their estimates apply to $\kappa_0^{2iu}$, the kernel of $\cL^{-iu}$, 
for $u$ real.

For each $c\geq 0$ and every positive integer $d$, we set   
\begin{equation}\label{kernel Riesz}
\br_{c}^{d}(x\cdot o)
= L(x)^*\big[\nabla^d \kappa_{c}^{d}\big](x\cdot o) \quant x\in G \setminus K;
\end{equation}
$\br_c^d$ is a $(\frp^*)^{\otimes d}$-valued smooth function on $\BX\setminus\set{o}$.
We recall that covariant differentiation on $\BX$ 
has a simple expression in terms of left invariant derivatives on~$G$~\cite[p. 264]{A1}.  Thus, 
\begin{equation}\label{f: covdiff}
\br_c^d (x\cdot o) (Z_1,\ldots,Z_d)
= Z_1\cdots Z_d\kappa_c^d(x) \quant x\in G \quant Z_1,\ldots,Z_d\in\frp.
\end{equation} 

\begin{lemma} \label{l: kernel Riesz} 
Suppose that $c\geq 0$ and that $d$ is a nonnegative integer.  Then 
\begin{equation} \label{f: conv Riesz}
L(x)^*\, \big[\shRiesz f\big](x\cdot o)
= f\ast \br_c^d (x\cdot o) 
\end{equation}
whenever $x\cdot o$ does not belong to $\supp f$.  
Furthermore, for every $\cZ$ in $\frp^d$
\begin{equation} \label{f: conv Riesz II}
\scalarshRiesz f(x\cdot o)
= f\ast (Z_1\cdots Z_d \kappa_c^d) (x\cdot o).
\end{equation}
\end{lemma}

\begin{proof}
Suppose that $f$ is in $C^\infty_c(\BX)$ and that $x\cdot o$ does not belong to $\supp f$.
Then
\begin{align*}
\shRiesz f(x\cdot o)
& = \nabla_x^d \int_G  f(y\cdot o)\,\, \kappa_c^d(y^{-1}x\cdot o) \wrt y \\ 
& = \int_G  f(y\cdot o)\,\, \nabla_x^d\big[ \kappa_c^d(y^{-1}x\cdot o)\big] \wrt y.
\end{align*}
Since the map $L(y)$ is an isometry of $\BX$,  
\begin{align*}
\nabla_x^d \big[\kappa_c^d(y^{-1}x\cdot o)\big]
& = \nabla^d \big[L(y^{-1})\kappa_c^d\big](x\cdot o)  \\
& = L(y^{-1})^*\,\big[\nabla^d \kappa_c^d\big](y^{-1}x\cdot o).  
\end{align*}
Thus,
\begin{align*}
\shRiesz f(x\cdot o)
& = \int_Gf(y\cdot o)\,L(y^{-1})^*\,\big[\nabla_x^d \kappa_c^d\big](y^{-1}x\cdot o) \wrt y\\
& = L(x^{-1})^* \int_G f(y\cdot o)\,  \br_c^d (y^{-1} x\cdot o) \wrt y;
\end{align*}
we have used the fact that $L(y^{-1})^* = L(x^{-1})^*L(y^{-1}x)^*$ in the last inequality 
above. The proof of \eqref{f: conv Riesz} is complete.

Formula \eqref{f: conv Riesz II} follows from \eqref{f: conv Riesz}, 
the definition of scalar Riesz transform \eqref{f: scalar Riesz}, and~\eqref{f: covdiff}.
\end{proof} 

\noindent
For technical reasons, which will be apparent shortly, we shall need to consider 
another tensor valued operator related to the Riesz transforms, namely $\nabla^d \cL^{1/2}$.  
The following lemma will be useful in the proof of Theorem \ref{t: RT}~\rmiii,
as will be Lemma~\ref{l: normL^{-k}} below.  

\begin{lemma}\label{l: B4Bc}
For every positive integer $d$ there exists a constant $C$ such that 
$$
\bignorm{\nabla^d \cL^{1/2} f}{L^1((4B)^c;T^d)}
\leq  C\, r_B^{-d-1}\, \bignorm{f}{\lu{B}} \quant f\in C^\infty_c(B) 
$$
for all balls $B$ of radius $r_B\le 1$.
\end{lemma}

\begin{proof}
Recall that
$$
\bignorm{\nabla^d \cL^{1/2} f}{L^1((4B)^c;T^d)}
= \int_{(4B)^c} \bigmod{\nabla^d \cL^{1/2} f(x\cdot o)}_{x\cdot o} \wrt \mu(x\cdot o),
$$
where $\mod{\nabla^d \cL^{1/2} f(x\cdot o)}_{x\cdot o}$ denotes the norm
of the covariant tensor $\nabla^d \cL^{1/2} f(x\cdot o)$.  
Since $L(x)^*$ is an isometry between covariant tensors at $x\cdot o$ and covariant
tensors~at~$o$, 
$$
\begin{aligned}
\bigmod{\nabla^d \cL^{1/2} f(x\cdot o)}_{x\cdot o}
& = \bigmod{L(x)^*[\nabla^d \cL^{1/2} f(x\cdot o)]}_{o} \\
& = \sup_{\mod{\cZ} \leq 1} \, \bigmod{L(x)^* \big[\nabla^d\,
\cL^{1/2} f(x\cdot o)\big](Z_1,\ldots,Z_d)}. 
\end{aligned}
$$
For every $d$-tuple of vectors $\cZ = (Z_1,\ldots,Z_d)$ in the unit ball of 
$\frp$, consider the scalar operator $\cS_{\cZ}^d$, defined by
$$
\cS_{\cZ}^d\,f(x\cdot o)
= L(x)^* \big[\nabla^d\,\cL^{1/2} f(x\cdot o)\big](Z_1,\ldots,Z_d).
$$
Thus, to prove the lemma it suffices to show that there exists a constant $C$ such that 
$$
\Bignorm{\sup_{\mod{\cZ} \leq 1} \, \mod{\cS_{\cZ}^df}}{\lu{(4B)^c}}
\leq C\, r_B^{-d-1}\, \bignorm{f}{\lu{B}}.  
$$
By arguing much as in in the proof of Lemma~\ref{l: kernel Riesz}, 
it is straightforward to check that
$$
\cS_{\cZ}^d\,f(x\cdot o)
= f\ast Z_1\cdots Z_d \kappa_0^{-1}(x\cdot o),
$$
where $\kappa_0^{-1}$ is defined in \eqref{f: kappacz}.  
For the sake of brevity, for the duration of this proof, we write 
$s_{\cZ}^d$ instead of $Z_1\cdots Z_d \kappa_0^{-1}$.
Thus, it suffices to show that
$$
\int_{(3B)^c} \sup_{\mod{\cZ}\leq 1} \bigmod{s_\cZ^d(x\cdot o)}  \wrt \mu(x\cdot o)
\leq  C\, r_B^{-d-1}.
$$
We write the integral as the sum of the integrals over the annulus $B_3(o) \setminus 3B$ 
and over $B_3(o)^c$, and estimate them separately.
\par

To estimate the first integral, we observe that, by \cite[Remark~4.2.3~\rmiii]{AJ}, there 
exists a constant $C$, independent of $Z_1,\ldots,Z_d$ in the unit ball of $\frp$, such that 
$$
\bigmod{s_\cZ^d(x\cdot o)} 
\leq C \,  \mod{x\cdot o}^{-1-n-d} 
\quant x\cdot o \in B_3(o).  
$$ 
Therefore
$$
\int_{B_3(o)\setminus 3B} \, \sup_{\mod{\cZ} \leq 1} \,\bigmod{s_\cZ^d(x\cdot o)} 
\wrt \mu(x\cdot o)
\leq  C\, r_B^{-d-1}.
$$

To estimate the second integral, we observe that, by \cite[Thm 4.2.2]{AJ}, there 
exists a constant $C$, independent of $Z_1,\ldots,Z_d$ in the unit ball of $\frp$, such that 
$$
{\bigmod{s_\cZ^d(x\cdot o)}}
\leq C\,  \big(1+\mod{x\cdot o}\big)^{-\mod{\Sigma_0^+}-1-\ell/2}\,
    \vp_0(x\cdot o)\,\, \e^{-\mod{\rho}\, \mod{x\cdot o}}.  
$$
Then we integrate in polar co-ordinates \eqref{intCartan}; 
by combining this estimate with estimates \eqref{delta} for the density $\de$
and \eqref{phi0} for $\vp_0$, we get
\begin{align*}
\int_{B_3(o)^c} \, \sup_{\mod{\cZ} \leq 1}\bigmod{s_\cZ^d(x\cdot o) }\wrt \mu(x\cdot o)
& =    C\, \int_{\frb_3^c} \, \sup_{\mod{\cZ} \leq 1}\bigmod{s_\cZ^d(\exp H\cdot o)}\, 
         \de(H) \wrt H\\
& \leq C\, \int_{\frb_3^c} \mod{H}^{-1-\ell/2}\, \e^{\rho(H)-\mod{\rho}\mod{H}} \wrt H.
\end{align*}
Denote by $(H_1,\ldots,H_\ell)$ the coordinates of $H$
 with respect to an orthonormal basis of~$\fra$ such that 
$H_1=\rho(H)/\mod{\rho}$. Then the latter integral is dominated by 
$$
\int_{\frb_3^c} \mod{H}^{-1-\ell/2}\, \e^{\mod{\rho}(H_1-\mod{H})} 
\wrt H_1\cdots \wrt H_\ell,
$$
which is easily seen to converge. This concludes the proof of the lemma.
\end{proof}

\begin{lemma}\label{l: normL^{-k}} 
Suppose that $k$ is a positive integer.
For every $X^k$-atom $A$ with support contained in $B$,
the support of $\cL^{-k}A$ is contained in $\OV B$. 
Furthermore, there exists a positive constant $C$, independent of $A$, such that 
$$
\bignorm{\cL^{-k}A}{2}
\leq C\, r_B^{2k}\, \bigmod{B}^{-1/2}.
$$
\end{lemma}

\begin{proof}
The support of $\cL^{-k}A$ is contained in $\OV B$ by \cite[Remark 3.5]{MMV3}. 
Denote by $\la_1(B)$ the smallest eigenvalue of the Dirichlet Laplacian on $B$. 
By Faber-Krahn's inequality~\cite{Gr1} there exists a positive constant $C$, independent of
the ball $B$, such that $\la_1(B)\ge C r_B^{-2}$. 
Hence the desired conclusion for $k=1$ follows from \cite[Corollary 3.3]{MMV5}. 
The general case follows from this by a straightforward induction argument. 
\end{proof}

\noindent 
We are now in position to prove Theorem \ref{t: RT}. 

\begin{proof} 
For the sake of simplicity, for the duration of this proof we shall denote the
kernel $Z_1\cdots Z_d \kappa_c^d$ of $\scalarshRiesz$ (see Lemma~\ref{l: kernel Riesz} above) 
simply by $\kappa$, and set
$$
A_j
:= \{x\cdot o\in \BX: j \leq \mod{x\cdot o} < j+2 \}.
$$
In view of \cite[Theorem 4.1]{MM} and the translation invariance of $\scalarshRiesz$,
to prove~\rmi\ it suffices to show that 
$$
\sup_{a} \, \bignorm{\scalarshRiesz \,a}{H^1}
< \infty, 
$$
where the supremum is taken over all $H^1$-atoms $a$ with support contained in $B_R(o)$
for some $R\leq 1$.  We analyse the cases where $R\geq 10^{-1}$ and $R< 10^{-1}$  
separately. 

Suppose first that $R\geq 10^{-1}$.  We consider a partition of unity on $\BX$ of the form
\begin{equation} \label{f: partition of unity I}
1 = \vp + \sum_{j=1}^\infty \psi_j,
\end{equation} 
where $\vp$ and $\psi_j$ are smooth $K$--invariant functions on $\BX$, 
the support of $\vp$ is contained in $B_2(o)$, and the support of $\psi_j$
is contained in the annulus $A_j$.  Then we write
$$
\scalarshRiesz a
= a\ast(\vp\kappa) + \sum_{j=1}^\infty a\ast(\psi_j\kappa).  
$$
We denote by $\bignorm{\vp\kappa}{Cv_2}$ the norm of the convolution operator
$f\mapsto f \ast (\vp\kappa)$, acting on $\ld{\BX}$.  Observe that $\bignorm{\kappa}{Cv_2}
< \infty$, because $\scalarshRiesz$ is bounded on $\ld{\BX}$.  Since $Cv_2(\BX)$
is a $C_c^\infty(\KGK)$-module, $\bignorm{\vp\kappa}{Cv_2}< \infty$. 
Thus,
\begin{equation} \label{f: local two part}
\begin{aligned}
\bignorm{a\ast(\vp\kappa)}{H^1} 
& \leq \bigmod{B_{R+2}(o)}^{1/2} \, \bignorm{a}{2} \, \bignorm{\vp\kappa}{Cv_2} \\  
& \leq \sqrt{\frac{\bigmod{B_{R+2}(o)}}{\bigmod{B_R(o)}} }
      \, \bignorm{\vp\kappa}{Cv_2} \\ 
& \leq C \bignorm{\vp\kappa}{Cv_2};
\end{aligned}
\end{equation} 
in the last inequality we have used the assumption $R \geq 10^{-1}$
and the local doubling condition.  
Furthermore, by Lemma~\ref{l: est on annuli} (in the case where $R+\be >1$), 
$$
\begin{aligned}
\bignorm{a\ast(\psi_j\kappa)}{H^1} 
& \leq (R+j+2) \, \bigmod{B_{R+j+2}(o)}^{1/2} \, \bignorm{\psi_j\kappa}{2}. 
\end{aligned}
$$
To estimate the $L^2$ norm of $\psi_j\kappa$ 
observe that \cite[Thm 4.2.2]{AJ} and estimate \eqref{phi0} imply that there 
exists a constant $C$ such that 
$$
\mod{\kappa(x\cdot o)}
\leq C\, \big(1+|x\cdot o|\big)^{(d-\ell-1)/2}\, \, 
\e^{-\rho(A^+(x))-\mod{x\cdot o}\sqrt{c^2+|\rho|^2}}.  
$$
By integrating in Cartan co-ordinates, and using the estimate above and the
estimate \eqref{delta} of the density function $\de$, we see that 
\begin{equation}\label{kinftyj}
\begin{aligned}
\bignormto{\psi_j \kappa}{2}2
& \leq C \, \int_{A_j}  |H|^{d-\ell-1}\,  \e^{-2|H|\sqrt{c^2+|\rho|^2}}\, \wrt H\\
& \leq C \, j^{d-2}\,\e^{-2j\sqrt{c^2+|\rho|^2}}.
\end{aligned}
\end{equation}
By combining the estimates above, we obtain that 
$$
\begin{aligned}
\bignorm{\scalarshRiesz a}{H^1}
& \leq \bignorm{a\ast(\vp\kappa)}{H^1} + \sum_{j=1}^\infty \, \bignorm{a\ast(\psi_j\kappa)}{H^1} \\
& \leq C \bignorm{\vp\kappa}{Cv_2}
      + C\,\sum_{j=1}^\infty \, j \, \bigmod{B_{j+3}(o)}^{1/2}\, j^{d/2-1}\,
        \e^{-j\sqrt{c^2+|\rho|^2}}.
\end{aligned}
$$
We use the estimate $\mod{B_{j+3}(o)} \leq C\, j^{\ell-1} \,\e^{2\mod{\rho}j}$, and conclude that 
$$
\bignorm{\scalarshRiesz a}{H^1}
\leq C \bignorm{\vp\kappa}{Cv_2}
      + C\,  \sum_{j=1}^\infty  \, j^{(d+\ell-1)/2}\,\e^{j(\mod{\rho}-\sqrt{c^2+|\rho|^2})},
$$
which is easily seen to be finite (and independent of $a$).   

Next suppose that $R < 10^{-1}$.  We denote by $D_{2^hR}$ the dyadic annulus 
$$
\{x\cdot o\in \BX: 2^{h-1} R \leq \mod{x\cdot o} < 2^{h+1} R\},
$$ 
and consider a partition of unity on $\BX$ of the form
\begin{equation} \label{f: partition of unity II}
1 = \phi +  \sum_{h=1}^N \eta_h + \psi_0 + \sum_{j=1}^\infty \psi_j,
\end{equation} 
where $\psi_j$, $j=1,2,\ldots$, is as in \eqref{f: partition of unity I},  
$\phi$ is a $K$--invariant function on $\BX$ with support contained in $B_{2R}(o)$,
$\eta_h$ are smooth $K$--invariant functions on $\BX$, with support contained in $D_{2^hR}$,
$N$ is the least integer for which $2^{N+1}R > 10^{-1}$, and the support of 
$\psi_0$ is contained in the annulus 
$\{ x\cdot o \in \BX:  10^{-1} \leq \mod{x\cdot o} \leq 2 \}$.  We also require that 
there exists a constant $C$ such that 
$$
\bigmod{\nabla\eta_h(x\cdot o)}
\leq C \, (2^hR)^{-1}
\quant h \in \{1,\ldots, N\}.  
$$  
It is important to keep in mind that $\phi$ and $\eta_j$ depend on~$R$. 

By arguing \emph{verbatim} as above, we may prove that the $\hu{\BX}$ norm 
of $\sum_{j=1}^\infty a\ast (\psi_j\kappa)$ is uniformly bounded with respect to $R$
in $(0,10^{-1}]$, and that the same is true of $a\ast (\psi_0\kappa)$.  
Also, much as for the estimate of the $\hu{\BX}$ norm of $a\ast (\vp\kappa)$ above, 
\begin{equation} \label{f: local R part}
\begin{aligned}
\bignorm{a\ast(\phi\kappa)}{H^1} 
& \leq \bigmod{B_{3R}(o)}^{1/2} \, \bignorm{a}{2} \, \bignorm{\phi\kappa}{Cv_2} \\  
& \leq \sqrt{\frac{\bigmod{B_{3R}(o)}}{\bigmod{B_R(o)}} }
      \, \bignorm{\phi\kappa}{Cv_2} \\ 
& \leq C \bignorm{\kappa}{Cv_2};
\end{aligned}
\end{equation}
in the last inequality we have used the local doubling condition, and  the fact
that multiplication by $\phi$ is a bounded operator on $Cv_2(\BX)$, with norm independent
of $R$, as long as $R$ stays bounded.  

Thus, to conclude the proof of~\rmi\ it suffices to show that the $\hu{\BX}$ norm of 
$\sum_{h=1}^N a\ast (\eta_h\kappa)$ is uniformly bounded with respect to $R$
in $(0,10^{-1}]$.  
Since the support of $a\ast (\eta_h\kappa)$ is contained in $B_1(o)$,
we may apply the first estimate in Lemma~\ref{l: est on annuli}~\rmi, and conclude that  
\begin{equation} \label{f: est aastetajkappa}
\bignorm{a\ast (\eta_h\kappa)}{H^1}
\leq \bigmod{B_{R+2^{h+1}R}(o)}^{1/2} \, 
      \min\big(\bignorm{\eta_h\kappa}{2}, C\, R\, \bignorm{\nabla(\eta_h\kappa)}{2} \big).  
\end{equation} 
By \cite[Remark~4.2.3~\rmiii]{AJ}, there exists a constant $C$ such that 
$$
\bigmod{\kappa(x\cdot o)}
\leq C\,  |x|^{-n}  
\qquad 
\bigmod{\nabla \kappa(x\cdot o)}
\leq C\,  |x|^{-n-1} 
\quant x \in G: \mod{x\cdot o} \leq 1.
$$
This, and the fact that the support of $\eta_h\kappa$ is contained in $D_{2^hR}$ imply that 
$$
\begin{aligned}
\bigmod{\nabla (\eta_h\kappa)(x\cdot o)}
& \leq C\,\big[(2^hR)^{-1}\,\mod{\kappa(x)} + \bigmod{\nabla \kappa(x)}\big]  \\
& \leq C\,\big[(2^hR)^{-1}\,\mod{x\cdot o}^{-n} + \mod{x\cdot o}^{-n-1}\big]
\quant x\in D_{2^hR}.  
\end{aligned}
$$
Therefore
$$
\begin{aligned}
\bignorm{\nabla(\eta_h\kappa)}{2} 
& \leq C\, (2^hR)^{-1} \, \Big[\int_{D_{2^hR}} |x|^{-2n} \wrt x \Big]^{1/2}
         + \Big[\int_{D_{2^hR}} |x|^{-2n-2} \wrt x \Big]^{1/2}\\
& \leq C\, (2^hR)^{-n-1} \, (2^hR)^{n/2}\\
& =    C\, (2^hR)^{-n/2-1}.
\end{aligned}
$$
This estimate, \eqref{f: est aastetajkappa}, and the fact that 
$|B_{(2^{h+1}+1)R}| \leq C \, (2^hR)^n$ imply that there exists a constant $C$, independent of $a$,
such that 
$$
\bignorm{a \ast (\eta_h\kappa)}{H^1}
\leq  C\, 2^{-h}, 
$$
so that $\sup_{R\leq 10^{-1}} \bignorm{\sum_{h=1}^N a\ast (\eta_h\kappa)}{H^1} < \infty$, and
the proof of \rmi\ is complete.    

Next we prove \rmii, i.e. that $\shRiesz$ is bounded from $\hu{\BX}$ to
$\lu{\BX;T^d}$.  A careful examination of the proof of \rmi\
reveals that there exists a constant $C$, independent of $Z_1,\ldots,Z_d$
in the unit ball of $\frp$, such that 
\begin{equation} \label{f: sup Riesz}
\Bignorm{\sup_{\mod{\cZ} \leq 1} \,\bigmod{\scalarshRiesz \,a}}{L^1}
\leq C.  
\end{equation}
As in the proof of Lemma~\ref{l: B4Bc} we use the fact that
$L(x)^*$ is an isometry between covariant tensors at the point $x\cdot o$ 
and covariant tensors at $o$, and conclude that 
$$
\bigmod{\shRiesz a(x\cdot o)}_{x\cdot o}
= \sup_{\mod{\cZ} \leq 1} \, \bigmod{\scalarshRiesz a(x\cdot o)}.  
$$
The required estimate follows directly from this and \eqref{f: sup Riesz}. 

Finally, we prove \rmiii. 
If $d$ is even, then $\floor{(d+1)/2}=d/2$ and the result is already known
\cite[Theorem 5.2]{MMV3}. 
Thus, we only need to consider the case when $d$ is odd, for which $\floor{(d+1)/2}=(d+1)/2$. 
By \cite[Corollary 6.2 and Proposition 6.3]{MMVdual} and the translation invariance
of $\Riesz$, it suffices to prove that 
\begin{equation} \label{f: aim}
\sup_A \bignorm{\Riesz A}{L^1(\BX;T^d)}
< \infty,
\end{equation}
where the supremum is taken over all $X^{(d+1)/2}$-atoms $A$ 
supported in balls centred at~$o$.
Given such an atom $A$, denote by $B_R(o)$ the ball associated to it.  
Observe that
\begin{equation}\label{f: local+global}
\bignorm{\Riesz A}{L^1(\BX;T^d)} 
= \bignorm{\mod{\Riesz A}}{L^1(4B)}  + \bignorm{\mod{\Riesz A}}{L^1((4B)^c)} .
\end{equation}
We shall estimate the two summands on the right hand side separately. 
Clearly
$$
\begin{aligned}
\bignorm{\mod{\Riesz A}}{L^1(4B)}
& \leq \mod{4B}^{1/2} \, \bignorm{\mod{\Riesz A}}{L^2(4B)} \\
& \leq C\,\sqrt{\mod{4B}/\mod{B}} \\
& \leq C;
\end{aligned}
$$
here we have applied the $L^2$-boundedness of $\Riesz$, the 
size property of $A$ and the local doubling property of $\mu$.

To estimate the second summand in \eqref{f: local+global} we write
$$
\Riesz A 
= \nabla^d \cL^{1/2}(\cL^{-(d+1)/2}A).
$$
By Schwarz's inequality and Lemma \ref{l: normL^{-k}}
there exists a constant $C$, independent of $A$, such that 
$$
\begin{aligned}
\bignorm{\cL^{-(d+1)/2}A}{\lu{B}}
& \leq \mod{B}^{1/2} \, \bignorm{\cL^{-(d+1)/2}A}{\ld{B}} \\
& \leq \mod{B}^{1/2} \, C\,  R^{d+1} \, \mod{B}^{-1/2} \\
& \leq  C\,  R^{d+1}.
\end{aligned}
$$
Now, Lemma~\ref{l: B4Bc} and this estimate imply that
\begin{equation*} 
\begin{aligned}  
\bignorm{\Riesz A}{\lu{(4B)^c;T^d}} 
& =    \bignorm{\mod{\nabla^d\cL^{1/2} \, (\cL^{-(d+1)/2}A)}}{\lu{(4B)^c}} \\
& \leq C\,  R^{-d-1} \, \bignorm{\cL^{-(d+1)/2}A}{L^1(B)} \\   
& \leq C.
\end{aligned}
\end{equation*}
This concludes the proof for odd $m$. 
\end{proof}

We notice that, by interpolation, Theorem \ref{t: RT} implies the $L^p$ boundedness 
of $\cR^d_c$ and $\cR^d$ for every $c>0$ and $p\in (1,2]$ 
(see \cite{CMM, MMV} for interpolation properties of $H^1(\BX)$ and $X^k(\BX)$). 

\section{Unboundedness on $\hu{\BX}$ of Riesz potentials 
and Riesz transform}   \label{s: negative results}

In this section we prove that the Riesz potentials $\cL^{-\si/2}$, 
$\si>0$, and the Riesz transform $\cR^1$ are unbounded from $\hu{\BX}$ to $\lu{\BX}$. 
Thus the endpoint result in Theorem~\ref{t: RT}~\rmiii\ is~sharp.

\begin{theorem}\label{t: notH^1L^1}
The operators $\cL^{-\si/2}$, $\si>0$, do not map $\hu{\BX}$ to $\lu{\BX}$  
and the Riesz transform $\cR^1$ does not map $\hu{\BX}$ to $\lu{\BX; T^1}$.
\end{theorem}

\noindent
The proof of this theorem requires some Harnack type estimates, 
which will be established in the next lemma.
For each positive real number $R$, denote by $H_R$ the element in the positive Weyl chamber 
$\mathfrak{a}^+$ such that $\mod{H_R}=R$ and 
$$
\prodo{H_R}{H}
= R \, \frac{\rho(H)}{\mod{\rho}} 
\quant H \in \mathfrak{a}.
$$ 
Set $a_R = \exp H_R$.  Recall that $\kappa_0^\si$ is the convolution
kernel of $\cL^{-\si/2}$ (see formula \eqref{f: kappacz}).  

\begin{lemma} \label{l: harnack like}
For each $\vep >0$ the following hold: 
\begin{enumerate}
\item[\itemno1]
there exists a positive number $\eta_0$ such that 
$$
\sup_{B_\eta(y\cdot o)}\, \kappa_0^\si
\leq (1+\vep) \,  \inf_{B_\eta(y\cdot o)}\, \kappa_0^\si
\quant \eta\leq \eta_0 \quant y\cdot o\in B_2(o)^c;  
$$
\item[\itemno2]
for each $R>0$ there exists a neighbourhood $U$ of the identity in $K$ such that 
$$
\kappa_0^\si(a_Rua \cdot o) 
\leq (1+\vep)\,  \kappa_0^\si(a_Ra \cdot o) 
\quant u \in U \quant a \in \exp \frb_2^c. 
$$
\end{enumerate}
\end{lemma}

\begin{proof}
First we prove \rmi.  Suppose that $x_1\cdot o, x_2\cdot o$ are points in $B_\eta(y\cdot o)$.  
By the mean value theorem  
$$
\begin{aligned}
\kappa_0^\si (x_1\cdot o) - \kappa_0^\si(x_2\cdot o)
& \leq 2 \eta \, \sup_{B_\eta(y\cdot o)} \bigmod{\nabla{\kappa_0^\si}}.  
\end{aligned} 
$$
By \cite[Thm~4.2.2]{AJ}, there exists a constant $C$ such that 
$$
\bigmod{\nabla{\kappa_0^\si(x\cdot o)}}  
\leq C\, \kappa_0^\si(x\cdot o)
\quant x\cdot o \notin B_1(o).  
$$
Therefore for all $y\cdot o$ in $B_2(o)^c$ 
$$
\begin{aligned}
\kappa_0^\si (x_2\cdot o) 
& \geq \kappa_0^\si(x_1\cdot o) - 2\eta \, \sup_{B_\eta(y\cdot o)} \bigmod{\nabla{\kappa_0^\si}} \\ 
& \geq \kappa_0^\si(x_1\cdot o) - 2 \, \eta \, C \, \sup_{B_\eta(y\cdot o)} \, \kappa_0^\si. 
\end{aligned}
$$
By taking the supremum over all $x_1$ in $B_\eta(y\cdot o)$, we obtain that 
$$
\kappa_0^\si (x_2\cdot o) 
\geq (1- 2 \, \eta\, C) \, \sup_{B_\eta(y\cdot o)} \, \kappa_0^\si. 
$$
Now, if $\eta< 1/(2C)$, then we may take the infimum of both sides over all 
$x_2$ in $B_\eta(y\cdot o)$, and obtain the required conclusion with 
$\eta_0 = \vep/2C(1+\vep)$.

To prove \rmii, write $a_Rua \cdot o = u^{a_R}a_Ra \cdot o$,
where $u^{a_R}$ is short for $a_Rua_R^{-1}$. 
By Lemma~\ref{l: Huc},  
$$
\begin{aligned} 
\bigmod{A^+(u^{a_R}a_Ra)-A^+(a_Ra)}
& \leq \wrt (u^{a_R}\cdot o,o) \\
& =    \wrt \big(\exp(\Ad (a_R) X)\cdot o,o\big) \\
& \leq \bigmod{\Ad (a_R) X}.  
\end{aligned}
$$
Here $X$ is in $\frk$, $\exp X = u$, and $\mod{X} \leq s$ with $s$ small, for we assume that $u$
belongs to a small neighbourhood of the origin in $K$.  
Notice that 
$
\Ad(a_R) X = \e^{\ad H_R} X 
$
so that 
$$
\bigmod{\Ad(a_R) X}
\leq \e^{\norm{\ad H_R}{}} \, \bigmod{X}
\leq \e^{\norm{\ad H_R}{}} \, s.   
$$  
Therefore $\exp\big(A^+(u^{a_R}a_Ra)\big)\cdot o$ lies in the ball with centre $a_Ra \cdot o$   
and radius $\e^{\norm{\ad H_R}{}} \, s$.  Assume that the latter quantity is smaller 
than $\eta_0$, i.e. that $s < \vep\, \e^{-\norm{H_R}{}}/2C(1+\vep)$.  Then \rmi\ implies that  
$$
\begin{aligned}
\kappa_0^\si(a_Rua \cdot o) 
& = \kappa_0^\si(\exp(A^+(u^{a_R}a_Ra) \cdot o) \\
& \leq (1+\vep)\,  \kappa_0^\si(a_Ra \cdot o)  
\quant u \in U \quant a \in \exp\frb_2^c,
\end{aligned}
$$
as required to conclude the proof of \rmii, and of the lemma.  
\end{proof}

\noindent
We now prove Theorem~\ref{t: notH^1L^1}.

\begin{proof}
Fix $\vep >0$.  Consider the function 
$
f  = b \, \One_{B_{\eta_0}(o)} -b \, \One_{B_{\eta_0}(a_R^{-1}\cdot o)},
$
where $\eta_0$ is as in Lemma~\ref{l: harnack like}~\rmi, 
$b = \mu\big(B_{\eta_0}(o)\big)^{-1}$, and $\One_E$ denotes the characteristic
function of~$E$.  
Clearly $f$ is in $L^2(\BX)$, its integral vanishes and 
its support is contained in $\OV{B_{R+1}(o)}$.  Then $f$ belongs to $\hu{\BX}$, 
by Remark~\ref{rem: economical dec}. 

We shall prove that if $R$ is large enough, then $\cL^{-\si/2}f$ is not in $\lu{\BX}$.  
We observe that this implies that the Riesz transform $\cR^1$ does not map 
$\hu{\BX}$ into $\lu{\BX;T^1}$.
Indeed, by Cheeger's inequality, there exists a positive constant~$c$ such that 
$$
\bignorm{\mod{\cR^1f}}{1}
\ge c\, \bignorm{\cL^{-1/2} f}{1},
$$
and the right hand side is infinite.

We continue the proof of the fact that $\cL^{-\si/2}f$ is not in $\lu{\BX}$.  
Observe that
$$
\begin{aligned}
f\ast \kappa_0^\si (x\cdot o)
& =  b \int_{B_{\eta_0}(o)} \kappa_0^\si(y^{-1}x\cdot o) \wrt\mu(y\cdot o) 
       - b  \int_{B_{\eta_0}(a_R^{-1}\cdot o)} \kappa_0^\si(y^{-1}x\cdot o) \wrt\mu(y\cdot o) \\  
& =  b \int_{B_{\eta_0}(o)} \big[\kappa_0^\si(y^{-1}x\cdot o) 
       - \kappa_0^\si(y^{-1}a_R x\cdot o)\big] \wrt\mu(y\cdot o) \\
& =  b \int_{B_{\eta_0}(o)} \big[\kappa_0^\si(x^{-1}y\cdot o) 
       - \kappa_0^\si(x^{-1}a_R^{-1} y\cdot o)\big] \wrt\mu(y\cdot o).  
\end{aligned}
$$
We have used the fact that $\kappa_0^\si(v^{-1}\cdot o) = \kappa_0^\si(v\cdot o)$ 
in the last equality.  
By Lemma~\ref{l: harnack like}~\rmi, the last integrand above is bounded from below by 
$$
\begin{aligned}
\frac{1}{1+\vep} \, \kappa_0^\si(x^{-1}\cdot o) 
         - (1+\vep)\, \kappa_0^\si(x^{-1}a_R^{-1} \cdot o) \\
& \hskip-.5truecm=    \frac{1}{1+\vep} \,  \kappa_0^\si(x\cdot o) \, \Big[1 \, - (1+\vep)^2\, 
         \frac{\kappa_0^\si(a_Rx \cdot o)}{\kappa_0^\si(x\cdot o)} \Big].  
\end{aligned}
$$
We have used again the fact that $\kappa_0^\si(v^{-1}\cdot o) = \kappa_0^\si(v\cdot o)$ 
in the last equality.  

We now restrict $x$ to $U\cdot \exp\big(\frc_\de\cap \frb_{2R}^c\big)$, where $U$
is a small neighbourhood of the identity in $K$ (as in Lemma~\ref{l: harnack like}~\rmii), and,
for $\de$ in $(0,1)$, 
$\frc_\de$ denotes the proper subcone of the positive Weyl chamber~$\fra^+$, defined by
\begin{align*}
\frc_\de
= \set{H\in\mathfrak{a}^+:  \rho(H)  \geq  \de \,\mod{\rho} \, \mod{H}}. 
\end{align*}
Then Lemma~\ref{l: harnack like}~\rmii\ implies that 
$\kappa_0^\si(a_Rx \cdot o) \leq (1+\vep) \,\kappa_0^\si(a_Ra \cdot o)$ for all such $x$, and we are
left with the problem of estimating 
$$
\frac{1}{1+\vep} \,  \kappa_0^\si(a\cdot o) \, \Big[1 \, - (1+\vep)^3\, 
         \frac{\kappa_0^\si(a_Ra \cdot o)}{\kappa_0^\si(a\cdot o)} \Big]  
\quant a \in \exp\big[\frc_\de\cap \frb_{2R}^c\big]
$$
from below.  

A straightforward consequence of \cite[Theor. 4.2.2]{AJ}, and of the sharp estimate of 
the spherical function $\vp_0$, is that for every $\de$ close to $1$,  
there exist positive constants~$c$ and $C$ such that 
\begin{equation} \label{f: spherical functions on cones}
c \, \mod{H}^{(\si-\ell-1)/2}\, \e^{-\rho(H)- \mod{\rho}\mod{H}} 
\leq \kappa_0^\si \big(u\exp H\cdot o\big)
\leq C \, \mod{H}^{(\si-\ell-1)/2}\, \e^{-\rho(H)- \mod{\rho}\mod{H}} 
\end{equation}  
for all $u$ in $K$ and for all $H$ in $\frc_\de\cap \frb_1^c$.  
Therefore
\begin{equation} \label{f: est ratio}
\begin{aligned}
\frac{\kappa_0^\si(a_Ra\cdot o)}{\kappa_0^\si(a\cdot o)} 
&\leq C\, \Big(\frac{\mod{H_R+H}}{\mod{H}}\Big)^{(\si-\ell-1)/2}
    \, \e^{-\rho(H_R+H) - \mod{\rho} \, \mod{H_R+H} + \rho(H) + \mod{\rho}\, \mod{H}}\\
&\leq C\, \Big(\frac{\mod{H_R+H}}{\mod{H}}\Big)^{(\si-\ell-1)/2}
    \, \e^{-\mod{\rho} (R + \mod{H_R+H} - \mod{H} )}.  
\end{aligned}
\end{equation}  
Observe that 
$$
\frac{1}{2} 
\leq \frac{\mod{H_R+H}}{\mod{H}} 
\leq \frac{3}{2}
\quant H \in \frc_\de\cap \frb_{2R}^c,
$$
and that the exponential on the right hand
side of \eqref{f: est ratio} is dominated by $\e^{-\mod{\rho} R}$.  
By choosing $R$ large enough, we may conclude that 
$$
\frac{\kappa_0^\si(a_Ra\cdot o)}{\kappa_0^\si(a\cdot o)} 
< \vep 
\quant a \in \exp\big[\frc_\de\cap \frb_{2R}^c\big]. 
$$
Altogether, we have proved that for $R$ large enough
$$
\begin{aligned}
\int_{U\cdot \exp(\frc_{\de}\cap\frb_{2R}^c)} \mod{f\ast\kappa_0^\si} \wrt\mu
& \geq b \,\, \frac{1-(1+\vep)^3\vep}{1+\vep} \, \int_{U\cdot \exp(\frc_{\de}\cap\frb_{2R}^c)} 
    \kappa_0^\si  \wrt \mu.  
\end{aligned}
$$
It is not hard to prove that the last integral is equal to infinity. 
Indeed, integrate in Cartan co-ordinates, use the fact that $\kappa_0^\si$
is $K$--bi-invariant, and obtain that the last integral is equal to 
$$
\mu_K(U) \, \int_{\frc_{\de}\cap\frb_{2R}^c} \kappa_0^\si(\exp H \cdot o) \, \de(H) \wrt H
\geq c \,\mu_K(U) \, \int_{\frc_{\de}\cap\frb_{2R}^c} \mod{H}^{(\si-\ell-1)/2}\, 
   \e^{\rho(H)- \mod{\rho}\mod{H}} \wrt H.  
$$
In the last inequality we have used the fact that $\de(H)\ge c\,\e^{2\rho(H)}$ 
for some positive constant $c$ when $H$ is in $\frc_{\de}$ (see \eqref{delta}).
If the rank $\ell$ is equal to one, then $\rho(H) = \mod{\rho}\, \mod{H}$, 
$\frc_{\de}\cap\frb_r^c$ reduces to the half line $[r,\infty)$, 
and the integrand becomes $\mod{H}^{-1+\si/2}$, which is nonintegrable on $[r,\infty)$.  
If $\ell \geq 2$, then we  pass to polar co-ordinates in $\fra$ and see that 
the last integral is equal to
$$
c\int_0^{\arccos \de} \! \wrt\theta\, (\sin\theta)^{\ell-2} 
\int_{r}^\infty s^{(\si+\ell-3)/2}\, \e^{\mod{\rho}\, (\cos\theta-1)s} \wrt s,
$$ 
which is easily seen to diverge for all $\si\ge0$. 

This concludes the proof that $\cL^{-\si/2}f$ is not in $\lu{\BX}$.  
\end{proof}

\end{document}